\documentclass[11pt]{amsart}
\usepackage{amsmath}
\usepackage{amssymb} 

\def\beqnn{\begin{eqnarray*}}\def\eeqnn{\end{eqnarray*}}
\newtheorem{theorem}{Theorem}[section]

\newtheorem{proposition}[theorem]{Proposition}
\newtheorem{corollary}[theorem]{Corollary}

\theoremstyle{definition}

\newtheorem{remark}[theorem]{Remark}

\theoremstyle{question}

\numberwithin{equation}{section}

\begin{document}

\title{Generalized Hilbert series operators}


\author{Jianjun Jin}
\address{School of Mathematics Sciences, Hefei University of Technology, Xuancheng Campus, Xuancheng 242000, P.R.China}
\email{jinjjhb@163.com, jin@hfut.edu.cn}

\author{Shuan Tang}
\address{School of Mathematics Sciences, Guizhou Normal University, Guiyang 550001,
P.R.China} \email{tsa@gznu.edu.cn}


\thanks{The authors were supported by National Natural Science
Foundation of China(Grant Nos. 11501157, 12061022).}





\begin{abstract}
In this note we study the generalized Hilbert series operator $H_{\mu}$, induced by a positive Bore measure $\mu$ on $[0, 1)$, between weighted sequence spaces. We characterize the measures $\mu$ for which  $H_{\mu}$ is bounded between different sequence spaces. Finally, for certain special measures, we obtain the sharp norm estimates of the operators and establish some new generalized Hilbert series inequalities with the best constant factors.
\end{abstract}

\maketitle
\section{Introduction}

Let $p>1$ and let $\alpha$ be a real number.  We define the weighted sequence space $l_{\alpha}^{p}$ as
\begin{equation*}l_{\alpha}^{p}:=\left\{a=\{a_n\}_{n=1}^{\infty}: \|a\|_{p, \alpha}=(\sum_{n=1}^{\infty} n^{\alpha} |a_{n}|^p)^{\frac{1}{p}}<\infty\right\}.\end{equation*}
If $\alpha=0$, we will write $l_p$ and $\|a\|_{p}$ instead of $l_{\alpha}^p$ and $\|a\|_{p, \alpha}$, respectively.

The Hilbert series operator, induced by the Hilbert kernel $\frac{1}{m+n}$, is defined as
\begin{equation*} H(a)(m)=\sum_{n=1}^\infty \frac{a_n}{m+n}, \: a=\{a_n\}_{n=1}^{\infty}, \: m\in \mathbb{N}.
\end{equation*}
It is well known that $H$ is bounded from $l_p$ into itself and $\|H\|=\pi \csc{\frac{\pi}{p}},$ see \cite{HLP}. Here
$$\|H\|=\sup_{a(\neq \theta)\in l_p}\frac{\|Ha\|_{p}}{\|a\|_p}.$$

It is natural to ask whether the Hilbert operator is still bounded from the weighted sequence space $l_{\alpha}^p$ into itself.   We see that it is the case for certain weighted sequence spaces,  and have the following
\begin{proposition}\label{pro-1}
Let $p>1$. If  $-1<\alpha<p-1$, then $H$ is bounded from  $l_{\alpha}^{p}$ into itself, and $\|H\|_{\alpha}=\pi \csc{\frac{\pi(1+\alpha)}{p}}$, where
$$\|H\|_{\alpha}=\sup_{a(\neq \theta)\in l_{\alpha}^p}\frac{\|Ha\|_{p, \alpha}}{\|a\|_{p, \alpha}}.$$
\end{proposition}
\begin{remark}
This result is known in the literature, see \cite{Jin} for an equivalent form of Proposition \ref{pro-1}. We will establish an extension of this result in the last section.
\end{remark}

However, we find the Hilbert operator is not bounded from $l_{\alpha}^p$ into $l_{\beta}^{p}$, if $\alpha<\beta$ and $\alpha>-1$. Actually, if $\alpha<\beta$, let $\varepsilon>0$ and set
$a_n=(\frac{\varepsilon}{1+\varepsilon})^{\frac{1}{p}}n^{-\frac{\alpha+1+\varepsilon}{p}}.$
It is easy to see that $$\|a\|_{p, \alpha}=\frac{\varepsilon}{1+\varepsilon}\sum_{n=1}^{\infty}n^{-1-\varepsilon}<\frac{\varepsilon}{1+\varepsilon}(1+\int_{1}^{\infty}x^{-1-\varepsilon}dx)=1.$$

For $\alpha>-1$, we have
\begin{eqnarray}\|Ha\|_{p,\beta}^p&=&\frac{\varepsilon}{1+\varepsilon}\sum_{m=1}^{\infty}m^{\beta}\left (\sum_{n=1}^{\infty}\frac{1}{m+n}\cdot n^{-\frac{1+\alpha+\varepsilon}{p}}\right )^p\nonumber \\ &=&\frac{\varepsilon}{1+\varepsilon}
\sum_{m=1}^{\infty}m^{\beta-\alpha-1-\varepsilon}\left[\sum_{n=1}^{\infty}\frac{1}{m+n}\cdot \left(\frac{m}{n}\right)^{\frac{1+\alpha+\varepsilon}{p}}\right]^p \nonumber \\
&\geq & \frac{\varepsilon}{1+\varepsilon}
\sum_{m=1}^{\infty}m^{\beta-\alpha-1-\varepsilon}\left[ \int_{1}^{\infty} \frac{1}{m+x}\cdot \left(\frac{m}{x}\right)^{\frac{1+\alpha+\varepsilon}{p}}\,dx \right]^p \nonumber \\
&=&\frac{\varepsilon}{1+\varepsilon}
\sum_{m=1}^{\infty}m^{\beta-\alpha-1-\varepsilon}\left[ \int_{\frac{1}{m}}^{\infty} \frac{1}{1+t}\cdot \left(\frac{1}{t}\right)^{\frac{1+\alpha+\varepsilon}{p}}\,dt \right]^p \nonumber \\
&\geq &\frac{\varepsilon}{1+\varepsilon}\sum_{m=1}^{\infty}m^{\beta-\alpha-1-\varepsilon}\left[\int_{1}^{\infty} \frac{1}{1+t}\cdot \left(\frac{1}{t}\right)^{\frac{1+\alpha+\varepsilon}{p}}\,dt \right]^p \nonumber
\end{eqnarray}
If $H: l_{\alpha}^p\rightarrow l_{\beta}^p$ is bounded, then there exists a constant $C_1>0$ such that
\begin{equation}\label{c-1}C_1\geq\frac{\|Ha\|_{p, \beta}^p}{\|a\|_{p,\alpha}^p}\geq \frac{\varepsilon}{1+\varepsilon}\sum_{m=1}^{\infty}m^{\beta-\alpha-1-\varepsilon}\left[\int_{1}^{\infty} \frac{1}{1+t}\cdot \left(\frac{1}{t}\right)^{\frac{1+\alpha+\varepsilon}{p}}\,dt \right]^p .
\end{equation}

But when $\varepsilon<\beta-\alpha$, we see that $$\sum_{m=1}^{\infty}m^{\beta-\alpha-1-\varepsilon}=+\infty.$$
Hence we get that (\ref{c-1}) is a contradiction. This implies that the Hilbert operator is not bounded from $l_{\alpha}^p$ into $l_{\beta}^{p}$, if $\alpha<\beta$ and $\alpha>-1$.

Note that the Hilbert kernel can be written as
$$\frac{1}{m+n}=\int_{0}^{1}t^{m+n-1}dt.$$

Let $\mu$ be a positive Bore measure on $[0, 1)$, we define the generalized Hilbert series operator $H_{\mu}$ as
\begin{equation*} H_{\mu}(a)(m):=\sum_{n=1}^\infty \mu[m+n]a_n, \: a=\{a_n\}_{n=1}^{\infty}, \: m\in \mathbb{N},
\end{equation*}
where $$\mu[n]=\int_{0}^{1}t^{n-1}d\mu(t),\: n\in \mathbb{N}.$$

In this note, we first study the problem of characterizing the measures $\mu$ such that $H_{\mu}: l_{\alpha}^p\rightarrow l_{\beta}^p$ is bounded. We provide a sufficient and necessary condition of $\mu$ for which $H_{\mu}: l_{\alpha}^p\rightarrow l_{\beta}^p$ is bounded. It should be pointed out that there has been a lot of work in recent years on the action of the Hilbert operator and its generalizations in different analytic function spaces. See for example \cite{GP}, \cite{GM}, \cite{BW}, \cite{CGP}, \cite{GM-2}.

To state our first result, we introduce the notation of generalized Carleson measure on $[0, 1)$.  Let $s>0$, $\mu$ be a positive Borel measure on $[0,1)$. We say $\mu$ is a $s$-Carleson measure if there is a constant $C_2>0$ such that $$\mu([t, 1))\leq C_2 (1-t)^s$$ for all $t\in [0, 1)$.

We now state the first main result of this paper.
\begin{theorem}\label{main}
Let $p>1$. Let $\alpha, \beta$ be such that $-1<\alpha, \beta<p-1$. Then the following statements are equivalent:

{\bf (1)} $H_{\mu}: l_{\alpha}^p\rightarrow l_{\beta}^p$ is bounded.

{\bf(2)} $\mu$ is a $[1+\frac{1}{p}(\beta-\alpha)]$-Carleson measure on $[0, 1)$.

{\bf(3)} $\mu[n]=O(n^{-1-\frac{1}{p}(\beta-\alpha)})$.
\end{theorem}

We end this section by fixing some notations. We denote by $q$ the conjugate of $p$, i.e., $\frac{1}{p}+\frac{1}{q}=1$.  For two positive numbers $A, B$, we write
$A \preceq B$, or $A \succeq B$, if there exists a positive constant $C$ independent of $A$ and $B$ such that $A \leq C B$, or $A \geq C B$, respectively. We will write $A \asymp B$ if $A \preceq B$ and $A \succeq B$.
\section{Proof of Theorem \ref{main}}

In our proof of Theorem \ref{main}, we need the Beta function defined as follows.
$$B(u,v)=\int_{0}^{\infty}\frac{t^{u-1}}{(1+t)^{u+v}}\,dt,\: u>0,v>0.$$
It is known that
$$B(u,v)=\int_{0}^{1}t^{u-1}(1-t)^{v-1}\,dt=\frac{\Gamma(u)\Gamma{(v)}}{\Gamma(u+v)}. $$
and $B(u,v)=B(v,u)$, where $\Gamma$ is the Gamma function, defined as
$$\Gamma(x)=\int_{0}^{\infty}e^{-t} t^{x-1}\,dt,\: x>0.$$
For more detailed introduction to the Beta function and Gamma function, see \cite{W} .

For $-1<\alpha, \beta<p-1$, we define
\begin{equation*}
W_{\alpha, \beta}^{[1]}(n):=\sum_{m=1}^{\infty}\frac{1}{(m+n)^{1+\frac{1}{p}(\beta-\alpha)}}\cdot \frac{n^{\frac{1+\alpha}{q}}}{m^{1-\frac{1+\beta}{p}}},\: n\in \mathbb{N},
\end{equation*}
and
\begin{equation*}
W_{\alpha, \beta}^{[2]}(m):=\sum_{n=1}^{\infty}\frac{1}{(m+n)^{1+\frac{1}{p}(\beta-\alpha)}}\cdot \frac{m^{(q-1)(1-\frac{1+\beta}{p})}}{n^{\frac{1+\alpha}{p}}},\: m\in \mathbb{N}.
\end{equation*}

Since $-1<\alpha, \beta<p-1$, we see that
\begin{eqnarray}\label{w-11}
W_{\alpha, \beta}^{[1]}(n)&\leq & \int_{0}^{\infty}\frac{1}{(x+n)^{1+\frac{1}{p}(\beta-\alpha)}}\cdot \frac{n^{\frac{1+\alpha}{q}}}{x^{1-\frac{1+\beta}{p}}}\,dx \\
&=& B(\frac{1+\beta}{p}, 1-\frac{1+\alpha}{p})n^{\alpha}.\nonumber
\end{eqnarray}
Similarly, we can show that
\begin{equation}\label{w-22}
W_{\alpha, \beta}^{[2]}(m)\leq B(\frac{1+\beta}{p}, 1-\frac{1+\alpha}{p})m^{(1-q)\beta}.
\end{equation}

Now, we start to prove Theorem \ref{main}. We first show

{\bf (2)$\Rightarrow$(3)}.  We note that {\bf(3)} is obvious when $n=1$.  We get from integration by parts that, for $n(\geq 2)\in \mathbb{N}$,
\begin{eqnarray}
\mu[n]=\int_{0}^1 t^{n-1}d\mu(t)&=&\mu([0,1))-(n-1)\int_{0}^1 t^{n-2}\mu([0, t))dt \nonumber \\
&=& (n-1)\int_{0}^1 t^{n-2}\mu([t, 1))dt.\nonumber
\end{eqnarray}

Since $\mu$ is a $[1+\frac{1}{p}(\beta-\alpha)]$-Carleson measure on $[0, 1)$, then we see that there is a constant $C_3>0$ such that
$$\mu([t,1))\leq C_3 (1-t)^{1+\frac{1}{p}(\beta-\alpha)},$$
for all $t\in [0,1).$

It follows that
\begin{eqnarray}\mu[n] &\leq & C_3 (n-1)\int_{0}^1 t^{n-2}(1-t)^{1+\frac{1}{p}(\beta-\alpha)}dt\nonumber \\&=&C_3(n-1)\cdot\frac{\Gamma(n-1)\Gamma(2+\frac{1}{p}(\beta-\alpha))}{\Gamma(n+1+\frac{1}{p}(\beta-\alpha))}\nonumber \\ &\asymp & \frac{1}{n^{1+\frac{1}{p}(\beta-\alpha)}}.\nonumber \end{eqnarray}
Here we have used the fact that
$$\Gamma(x) = \sqrt{2\pi} x^{x-\frac{1}{2}}e^{-x}[1+r(x)],\,x>0, $$
where $|r(x)|\leq e^{\frac{1}{12x}}-1.$ Hence {\bf (2)$\Rightarrow$(3)} is true.

{\bf (3)$\Rightarrow$(1)}. Take $a=\{a_n\}_{n=1}^{\infty} \in l_{\alpha}^p$ and assume, without loss of generality, that $a_n \geq 0, \, n\in \mathbb{N}$.  By H\"{o}lder's inequality and (\ref{w-22}), we see from
$$\mu[m+n]=O\left(\frac{1}{(m+n)^{1+\frac{1}{p}(\beta-\alpha)}}\right)$$ that, for $m\in \mathbb{N}$,
\begin{eqnarray}\label{g1}
\lefteqn{\left |\sum_{n=1}^{\infty} \mu[m+n]a_n\right|\preceq\left |\sum_{n=1}^{\infty}\frac{a_n}{(m+n)^{1+\frac{1}{p}(\beta-\alpha)}}\right|}  \nonumber \\
&=&\sum_{n=1}^{\infty} \left\{[\frac{1}{(m+n)^{1+\frac{1}{p}(\beta-\alpha)}}]^{\frac{1}{p}}\cdot \frac{n^{\frac{1+\alpha}{pq}}}{m^{\frac{1}{p}(1-\frac{1+\beta}{p})}}\cdot a_n\right\}
\left\{[\frac{1}{(m+n)^{1+\frac{1}{p}(\beta-\alpha)}}]^{\frac{1}{q}}\cdot \frac{m^{\frac{1}{p}(1-\frac{1+\beta}{p})}}{n^{\frac{1+\alpha}{pq}}}\right\}
\nonumber \\
&\leq & \left[ \sum_{n=1}^{\infty}\frac{1}{(m+n)^{1+\frac{1}{p}(\beta-\alpha)}}\cdot \frac{n^{\frac{1+\alpha}{q}}}{m^{1-\frac{1+\beta}{p}}}\cdot a_n^{p}\right]^{\frac{1}{p}}
\left[\sum_{n=1}^{\infty}\frac{1}{(m+n)^{1+\frac{1}{p}(\beta-\alpha)}}\cdot \frac{m^{(q-1)(1-\frac{1+\beta}{p})}}{n^{\frac{1+\alpha}{p}}}\right]^{\frac{1}{q}}
\nonumber \\
&=&
[W_{\alpha, \beta}^{[2]}(m)]^{\frac{1}{q}}\left[ \sum_{n=1}^{\infty}\frac{1}{(m+n)^{1+\frac{1}{p}(\beta-\alpha)}}\cdot \frac{n^{\frac{1+\alpha}{q}}}{m^{1-\frac{1+\beta}{p}}}\cdot a_n^{p}\right]^{\frac{1}{p}}
\nonumber \\
&=&
[B(\frac{1+\beta}{p}, 1-\frac{1+\alpha}{p})]^{\frac{1}{q}} m^{-\frac{\beta}{p}}\left[\sum_{n=1}^{\infty}\frac{1}{(m+n)^{1+\frac{1}{p}(\beta-\alpha)}}\cdot \frac{n^{\frac{1+\alpha}{q}}}{m^{1-\frac{1+\beta}{p}}}\cdot a_n^{p}\right]^{\frac{1}{p}}. \nonumber
\end{eqnarray}

Consequently, we obtain from (\ref{w-11}) that
\begin{eqnarray}
\lefteqn{\|H_{\mu}a\|_{p, \beta}=\left[\sum_{m=1}^{\infty}m^{\beta}\left |\sum_{n=1}^{\infty} \mu[m+n]a_n\right|^p\right]^{\frac{1}{p}} \preceq \left[ \sum_{m=1}^{\infty}m^{\beta}\left|\sum_{n=1}^{\infty}\frac{a_n}{(m+n)^{1+\frac{1}{p}(\beta-\alpha)}}\right|^{p}\right]^{\frac{1}{p}}}\nonumber \\
&&\leq [B(\frac{1+\beta}{p}, 1-\frac{1+\alpha}{p})]^{\frac{1}{q}}\left[\sum_{m=1}^{\infty}\sum_{n=1}^{\infty}\frac{1}{(m+n)^{1+\frac{1}{p}(\beta-\alpha)}}\cdot \frac{n^{\frac{1+\alpha}{q}}}{m^{1-\frac{1+\beta}{p}}}\cdot a_n^{p}\right]^{\frac{1}{p}}
\nonumber \\
&&=[B(\frac{1+\beta}{p}, 1-\frac{1+\alpha}{p})]^{\frac{1}{q}}\left[\sum_{n=1}^{\infty}W_{\alpha, \beta}^{[1]}(n)a_n^{p}\right]^{\frac{1}{p}} \nonumber \\
&&\leq B(\frac{1+\beta}{p}, 1-\frac{1+\alpha}{p})\|a\|_{p, \alpha}.
\nonumber
\end{eqnarray}
This proves {\bf (3)$\Rightarrow$(1)}.

{\bf (1)$\Rightarrow$(2)}.  We need the following estimate given in \cite{Zh}.  Let $0<t<1$. For any $c>0$, we have
\begin{equation}\label{est}\sum_{n=1}^{\infty}n^{c-1}t^{2n}\asymp \frac{1}{(1-t^2)^c}.
\end{equation}

For $0<b<1$, we set $$\widetilde{a}_n=(1-b^2)^{\frac{1}{p}}n^{-\frac{\alpha}{p}}b^{\frac{2n}{p}},\; n\in \mathbb{N}.$$ Then we see from (\ref{est}) that $\|\widetilde{a}\|_{p, \alpha}\asymp 1.$ In view of the boundedness of $H_{\mu}: l_{\alpha}^p\rightarrow l_{\beta}^p$, we obtain that
\begin{eqnarray}
1 &\succeq & \|H_{\mu}\widetilde{a}\|_{p, \beta}^p =\sum_{m=1}^{\infty}m^{\beta}\left |\sum_{n=1}^{\infty}\widetilde{a}_n\int_{0}^1t^{m+n-1}d\mu(t) \right|^p\nonumber \\
&=&(1-b^2)\sum_{m=1}^{\infty}m^{\beta}\left[\sum_{n=1}^{\infty}n^{-\frac{\alpha}{p}}b^{\frac{2n}{p}}\int_{0}^{1}t^{m+n-1}d\mu(t)\right]^{{p}}
\nonumber \\
&\geq &(1-b^2)\sum_{m=1}^{\infty}m^{\beta}\left[\sum_{n=1}^{\infty}n^{-\frac{\alpha}{p}}b^{\frac{2n}{p}}\int_{b}^{1}t^{m+n-1}d\mu(t))\right]^{{p}} \nonumber \\
&\geq &(1-b^2)[\mu([b, 1))]^{p}\sum_{m=1}^{\infty}m^{\beta}\left(\sum_{n=1}^{\infty}n^{-\frac{\alpha}{p}}b^{\frac{2n}{p}}\cdot b^{m+n-1}\right)^{{p}}\nonumber \\
&=&(1-b^2)[\mu([b, 1))]^{p}\left(\sum_{m=1}^{\infty}m^{\beta}b^{m}\right)\left(\sum_{n=1}^{\infty}n^{-\frac{\alpha}{p}}b^{\frac{2n}{p}+n-1}\right)^{{p}}\nonumber \\
&\asymp & (1-b^2)[\mu([b, 1))]^{p}\cdot \frac{1}{(1-b^2)^{1+\beta}} \cdot\frac{1}{(1-b^2)^{p-\alpha}}.
\nonumber
\end{eqnarray}
This implies that
$$\mu([b, 1))\preceq (1-b^2)^{1+\frac{1}{p}(\beta-\alpha)},$$
for all $0<b<1$. It follows that $\mu$ is a $[1+\frac{1}{p}(\beta-\alpha)]$-Carleson measure on $[0, 1)$ and {\bf (1)$\Rightarrow$(2)} is proved.
The proof of Theorem \ref{main} is now finished.

\section{New generalized Hilbert series inequalities}
In this section, we consider certain $1$-Carleson measures and study a generalized Hilbert series operator induced by a bounded function on $[0, 1)$.
As applications, we establish some new generalized Hilbert series inequalities with the best constant factors.

Let $g$ be a non-negative and non-decreasing bounded function on $[0, 1)$.  We further assume that $\|g\|_{\infty}>0$ and set
$$\Lambda_g[n]:= \int_{0}^1 t^{n-1}g(t)\,dt, \, n\in \mathbb{N}.$$

We define the generalized Hilbert series operator $H_g$ as
$$H_g(a)(m)=\sum_{n=1}^{\infty}\Lambda_g[m+n]a_n =\sum_{n=1}^{\infty} a_n \int_{0}^1 t^{m+n-1}g(t)\,dt,\: a=\{a_{n}\}_{n=1}^{\infty},\: m\in \mathbb{N}.$$

\begin{remark}
When $g\equiv1$, $H_g$ becomes the classical Hilbert series operator.  We see from the fact that $g$ is a non-negative bounded function on $[0, 1)$ that $g(t)dt$ is  a $1$-Carleson measure on $[0, 1)$.  Then, by Theorem \ref{main}, we know that $H_{g}: l_{\alpha}^p\rightarrow l_{\alpha}^p$ is bounded if $-1<\alpha<p-1$.  Moreover, we shall show the following result.
\end{remark}

\begin{theorem}\label{main-1}
Let $p>1, -1<\alpha<p-1$. Let $g, H_g$ be as above.  Then we have  $H_{g}: l_{\alpha}^p\rightarrow l_{\alpha}^p$ is bounded, and $\|H_g\|_{\alpha}=\|g\|_{\infty}\pi \csc{\frac{\pi(1+\alpha)}{p}}$, where
$$\|H_g\|_{\alpha}=\sup_{a(\neq \theta)\in l_{\alpha}^p}\frac{\|H_ga\|_{p, \alpha}}{\|a\|_{p, \alpha}}.$$
\end{theorem}
\begin{remark}
Proposition \ref{pro-1} follows if we take $g\equiv1$.
\end{remark}
It follows from Theorem \ref{main-1} that
\begin{corollary}
Under the assumptions and with the notations of Theorem \ref{main-1}, we have the following generalized Hilbert inequality
\begin{equation}\label{last}\left[\sum_{m=1}^{\infty} m^{\alpha} \left(\sum_{n=1}^{\infty} a_n \int_{0}^1 t^{m+n-1}g(t)\,dt\right)^p \right]^{\frac{1}{p}} \leq \|g\|_{\infty}\pi \csc{\frac{\pi(1+\alpha)}{p}} \|a\|_{p, \alpha},\end{equation}
holds for all $a\in l_{\alpha}^p$, and the constant factor $\|g\|_{\infty}\pi \csc{\frac{\pi(1+\alpha)}{p}}$ in (\ref{last}) is the best possible.
\end{corollary}

\begin{proof}[Proof of Theorem \ref{main-1}]
For $a=\{a_n\}_{n=1}^{\infty} \in l_{\alpha}^p, \, a_n\geq 0, \, n \in \mathbb{N}$, by H\"{o}lder's inequality  and (\ref{w-22}), we obtain that, for $m\in \mathbb{N}$,
\begin{eqnarray}\label{g1}
\lefteqn{\left |\sum_{n=1}^{\infty} \Lambda_g[m+n]a_n\right|=\left|\sum_{n=1}^{\infty} a_n \int_{0}^1 t^{m+n-1}g(t)\,dt\right |}  \nonumber \\
&\leq&\|g\|_{\infty}\sum_{n=1}^{\infty} \left\{[\frac{1}{m+n}]^{\frac{1}{p}}\cdot \frac{n^{\frac{1+\alpha}{pq}}}{m^{\frac{1}{p}(1-\frac{1+\alpha}{p})}}\cdot a_n\right\}
\left\{[\frac{1}{m+n}]^{\frac{1}{q}}\cdot \frac{m^{\frac{1}{p}(1-\frac{1+\alpha}{p})}}{n^{\frac{1+\alpha}{pq}}}\right\}
\nonumber \\
&\leq &\|g\|_{\infty} \left[ \sum_{n=1}^{\infty}\frac{1}{m+n}\cdot \frac{n^{\frac{1+\alpha}{q}}}{m^{1-\frac{1+\alpha}{p}}}\cdot a_n^{p}\right]^{\frac{1}{p}}
\left[\sum_{n=1}^{\infty}\frac{1}{m+n}\cdot \frac{m^{(q-1)(1-\frac{1+\alpha}{p})}}{n^{\frac{1+\alpha}{p}}}\right]^{\frac{1}{q}}
\nonumber \\
&=& \|g\|_{\infty}
[W_{\alpha, \alpha}^{[2]}(m)]^{\frac{1}{q}}\left[ \sum_{n=1}^{\infty}\frac{1}{m+n}\cdot \frac{n^{\frac{1+\alpha}{q}}}{m^{1-\frac{1+\alpha}{p}}}\cdot a_n^{p}\right]^{\frac{1}{p}}
\nonumber \\
&\leq & \|g\|_{\infty}
[\pi\csc\frac{\pi(1+\alpha)}{p}]^{\frac{1}{q}} m^{-\frac{\alpha}{p}}\left[\sum_{n=1}^{\infty}\frac{1}{m+n}\cdot \frac{n^{\frac{1+\alpha}{q}}}{m^{1-\frac{1+\alpha}{p}}}\cdot a_n^{p}\right]^{\frac{1}{p}}. \nonumber
\end{eqnarray}
Here we have used the fact that $B(s, 1-s)=\pi \csc\pi s$ when $0<s<1.$

Consequently, we get from (\ref{w-11}) that
\begin{eqnarray}
\lefteqn{\|H_{g}a\|_{p, \alpha}=\left[\sum_{m=1}^{\infty}m^{\beta}\left |\sum_{n=1}^{\infty}\Lambda_g[m+n]a_n\right|^p\right]^{\frac{1}{p}}} \nonumber \\
&&\leq \|g\|_{\infty}[\pi\csc\frac{\pi(1+\alpha)}{p}]^{\frac{1}{q}}\left[\sum_{m=1}^{\infty}\sum_{n=1}^{\infty}\frac{1}{m+n}\cdot \frac{n^{\frac{1+\alpha}{q}}}{m^{1-\frac{1+\alpha}{p}}}\cdot a_n^{p}\right]^{\frac{1}{p}}
\nonumber \\
&&=\|g\|_{\infty}[\pi\csc\frac{\pi(1+\alpha)}{p}]^{\frac{1}{q}}\left[\sum_{n=1}^{\infty}W_{\alpha, \alpha}^{[1]}(n)a_n^{p}\right]^{\frac{1}{p}} \nonumber \\
&&\leq \|g\|_{\infty}\pi\csc\frac{\pi(1+\alpha)}{p}\|a\|_{p, \alpha}.
\nonumber
\end{eqnarray}
This proves that $H_{g}: l_{\alpha}^p\rightarrow l_{\alpha}^p$ is bounded and $\|H_g\|_{\alpha}\leq \|g\|_{\infty}\pi \csc{\frac{\pi(1+\alpha)}{p}}$.

Finally, we prove that $\|H_g\|_{\alpha}=\|g\|_{\infty}\pi \csc{\frac{\pi(1+\alpha)}{p}}$. For any $\varepsilon\in (0, \|g\|_{\infty})$, we see from the fact that $g$ is non-decreasing on $[0, 1)$ that there is a constant $j_{\varepsilon}\in (0, 1)$ such that
$$g(t)\geq \|g\|_{\infty}-\frac{1}{2}\varepsilon$$
for all $t\in[j_{\varepsilon},1).$ It follows that
\begin{eqnarray}\label{je}\Lambda_g[m+n]&\geq& (\|g\|_{\infty}-\frac{1}{2}\varepsilon)\int_{j_{\varepsilon}}^1 t^{m+n-1}\,dt=\frac{\|g\|_{\infty}-\frac{1}{2}\varepsilon}{m+n}(1-j_{\varepsilon}^{m+n}) \\
&=& \frac{\|g\|_{\infty}-\varepsilon}{m+n}\left[1+\frac{\varepsilon}{2(\|g\|_{\infty}-\varepsilon)}\right](1-j_{\varepsilon}^{m+n}).\nonumber\end{eqnarray}

For all $m\in \mathbb{N}$, since $j_{\varepsilon}^{m+n} \leq j_{\varepsilon}^{n}$, and $j_{\varepsilon}^{n} \rightarrow 0$($ n \rightarrow \infty$), we conclude from (\ref{je}) that there is a $\mathcal{N}=\mathcal{N}(\varepsilon) \in \mathbb{N}$ such that
\begin{equation}\Lambda_g[m+n]\geq \frac{\|g\|_{\infty}-\varepsilon}{m+n}\end{equation}
for all $n>\mathcal{N}$, and all $m\in \mathbb{N}.$

Let $\tau>0$, we set $\widehat{a}_n=0$ when $n\in [1, \mathcal{N}]$, $\widehat{a}_n=(\tau \mathcal{N}^{\tau})^{\frac{1}{p}}n^{-\frac{1+\alpha+\tau}{p}}$ when $n>\mathcal{N}.$ It is easy to see that
$$\|\widehat{a}\|_{p, \alpha}^p=\tau \mathcal{N}^{\tau}\sum_{n=\mathcal{N}+1}^{\infty}n^{-1-\tau}\leq \tau\mathcal{N}^{\tau} \int_{\mathcal{N}}^{\infty}x^{-1-\tau}\,dx=1.$$

Then it follows that
\begin{eqnarray}\label{eq-1}
\|H_g\|_{\alpha}&\geq& \|H_{g}\widehat{a}\|_{p, \alpha}=\left[\sum_{m=1}^{\infty}m^{\alpha}\left|\sum_{n=1}^{\infty}\Lambda_g[m+n]a_n\right|^p\right]^{\frac{1}{p}} \\
&\geq& (\|g\|_{\infty}-\varepsilon)(\tau \mathcal{N}^{\tau})^{\frac{1}{p}}\left[\sum_{m=1}^{\infty}m^{\alpha}\left |\sum_{n=\mathcal{N}+1}^{\infty}\frac{1}{m+n}\cdot n^{-\frac{1+\alpha+\tau}{p}}\right|^p\right]^{\frac{1}{p}} \nonumber \\
&\geq& (\|g\|_{\infty}-\varepsilon)(\tau \mathcal{N}^{\tau})^{\frac{1}{p}}\left[\sum_{m=1}^{\infty}m^{\alpha}\left |\int_{\mathcal{N}+1}^{\infty}\frac{1}{m+x}\cdot x^{-\frac{1+\alpha+\tau}{p}}\,dx\right|^p\right]^{\frac{1}{p}}
\nonumber \\
&=&  (\|g\|_{\infty}-\varepsilon)(\tau \mathcal{N}^{\tau})^{\frac{1}{p}}\left[\sum_{m=1}^{\infty}m^{-1-\tau}\left |\int_{\frac{\mathcal{N}+1}{m}}^{\infty}\frac{1}{1+t}\cdot t^{-\frac{1+\alpha+\tau}{p}}\,dt\right|^p\right]^{\frac{1}{p}}\nonumber.
\end{eqnarray}

It is clear that
\begin{eqnarray}\label{eq-2}
\lefteqn{\left[\sum_{m=1}^{\infty}m^{-1-\tau}\left |\int_{\frac{\mathcal{N}+1}{m}}^{\infty}\frac{1}{1+t}\cdot t^{-\frac{1+\alpha+\tau}{p}}\,dt\right|^p\right]^{\frac{1}{p}}}\\
&&\geq \left[\sum_{m=\mathcal{N}+1}^{\infty}m^{-1-\tau}\left |\int_{0}^{\infty}\frac{1}{1+t}\cdot t^{-\frac{1+\alpha+\tau}{p}}\,dt-\int_{0}^{\frac{\mathcal{N}+1}{m}}\frac{1}{1+t}\cdot t^{-\frac{1+\alpha+\tau}{p}}\,dt\right|^p\right]^{\frac{1}{p}}. \nonumber
\end{eqnarray}

On the other hand, when $\tau\in (0, p-1-\alpha)$, we have
\begin{equation}\label{eq-3}
D_{p,\alpha}(\tau):=\int_{0}^{\infty}\frac{1}{1+t}\cdot t^{-\frac{1+\alpha+\tau}{p}}\,dt=\pi \csc\frac{\pi(1+\alpha+\tau)}{p},
\end{equation}
and
\begin{eqnarray}\label{eq-4}
E_{p, \alpha}(\tau,m):&=&\int_{0}^{\frac{\mathcal{N}+1}{m}}\frac{1}{1+t}\cdot t^{-\frac{1+\alpha+\tau}{p}}\,dt\leq  \int_{0}^{\frac{\mathcal{N}+1}{m}}t^{-\frac{1+\alpha+\tau}{p}}\,dt \\
&=&\frac{p}{p-1-\alpha-\tau}\cdot \left(\frac{\mathcal{N}+1}{m}\right)^{\frac{p-1-\alpha-\tau}{p}}. \nonumber
\end{eqnarray}

By using the Bernoulli's inequality(see \cite{K}), (\ref{eq-3}) and (\ref{eq-4}), we see that
\begin{eqnarray}\label{eq-5}
\lefteqn{\left |\int_{0}^{\infty}\frac{1}{1+t}\cdot t^{-\frac{1+\alpha+\tau}{p}}\,dt-\int_{0}^{\frac{\mathcal{N}+1}{m}}\frac{1}{1+t}\cdot t^{-\frac{1+\alpha+\tau}{p}}\,dt\right|^p}
 \\
&&=[\pi \csc\frac{\pi(1+\alpha+\tau)}{p}]^{p}\left|1-\frac{E_{p, \alpha}(\tau, m)}{D_{p,\alpha}(\tau)}\right|^p\nonumber \\
&&\geq [\pi \csc\frac{\pi(1+\alpha+\tau)}{p}]^{p}\left[1-\frac{pE_{p, \alpha}(\tau, m)}{D_{p,\alpha}(\tau)}\right], \nonumber
\end{eqnarray}
and
\begin{eqnarray}\label{eq-6}
\lefteqn{\sum_{m=\mathcal{N}+1}^{\infty}m^{-1-\tau}\cdot\frac{pE_{p, \alpha}(\tau, m)}{D_{p, \alpha}(\tau)}}\\&&\quad \leq \frac{p^2(\mathcal{N}+1)^{\frac{p-1-\alpha-\tau}{p}}}{(p-1-\alpha-\tau)D_{p, \alpha}(\tau)} \sum_{m=\mathcal{N}+1}^{\infty} m^{-1-\tau-\frac{p-1-\alpha-\tau}{p}}\nonumber \\
&&\quad \leq  \frac{p^2(\mathcal{N}+1)^{\frac{p-1-\alpha-\tau}{p}}}{(p-1-\alpha-\tau)D_{p, \alpha}(\tau)} \int_{\mathcal{N}+1}^{\infty} x^{-1-\tau-\frac{p-1-\alpha-\tau}{p}}\,dx \nonumber \\
&&\quad =\frac{p^3(\mathcal{N}+1)^{-\tau}[D_{p, \alpha}(\tau)]^{-1}}{(p-1-\alpha-\tau)(p\tau+p-1-\alpha-\tau)}:=F_{p, \alpha}(\mathcal{N}, \tau). \nonumber
\end{eqnarray}

By (\ref{eq-2}), (\ref{eq-5}), (\ref{eq-6}), we obtain that
\begin{eqnarray}\label{eq-7}
\lefteqn{\left[\sum_{m=1}^{\infty}m^{-1-\tau}\left |\int_{\frac{\mathcal{N}+1}{m}}^{\infty}\frac{1}{1+t}\cdot t^{-\frac{1+\alpha+\tau}{p}}\,dt\right|^p\right]^{\frac{1}{p}}}\\
&&\geq \pi \csc\frac{\pi(1+\alpha+\tau)}{p}\left[\sum_{m=\mathcal{N}+1}^{\infty}m^{-1-\tau}- F_{p, \alpha}(\mathcal{N}, \tau)\right]^{\frac{1}{p}}\nonumber \\
&& \geq \pi \csc\frac{\pi(1+\alpha+\tau)}{p}\left\{[\tau(\mathcal{N}+1)^{\tau}]^{-1}- F_{p, \alpha}(\mathcal{N}, \tau)\right\}^{\frac{1}{p}}\nonumber \\
&&=\pi \csc\frac{\pi(1+\alpha+\tau)}{p}[\tau(\mathcal{N}+1)^{\tau}]^{-\frac{1}{p}}\left[1-\tau(\mathcal{N}+1)^{\tau}F_{p, \alpha}(\mathcal{N}, \tau)\right]^{\frac{1}{p}}. \nonumber
\end{eqnarray}
It follows from (\ref{eq-1}) that
\begin{equation}\label{eq-8}
\|H_g\|_{\alpha}\geq (\|g\|_{\infty}-\varepsilon)\pi \csc\frac{\pi(1+\alpha+\tau)}{p}\cdot [\mathcal{N}(\mathcal{N}+1)^{-1}]^{\frac{\tau}{p}}\left[1-\tau(\mathcal{N}+1)^{\tau}F_{p, \alpha}(\mathcal{N}, \tau)\right]^{\frac{1}{p}}.
\end{equation}

Take $\tau \rightarrow 0^{+}$ in (\ref{eq-8}),  we easily see that
\begin{eqnarray}
\|H_g\|_{\alpha}\geq (\|g\|_{\infty}-\varepsilon)\pi \csc\frac{\pi(1+\alpha)}{p},\nonumber
\end{eqnarray}
for any $\varepsilon \in (0, \|g\|_{\infty})$. It follows that  $\|H_g\|_{\alpha}\geq \|g\|_{\infty}\pi \csc\frac{\pi(1+\alpha)}{p}$. Hence $\|H_g\|_{\alpha}=\|g\|_{\infty}\pi \csc\frac{\pi(1+\alpha)}{p}$. Theorem \ref{main-1} is proved.

\end{proof}

\section*{Acknowledgments}
The author are grateful to the referee for his/her valuable suggestions which improve this paper.


\end{document}